\documentclass[11pt]{article}
\usepackage{latexsym}
\usepackage{amsmath}
\usepackage{latexsym,bm}
\usepackage{amsthm,color}
\usepackage{amssymb}
\usepackage{mathrsfs}
\usepackage{enumerate}
\usepackage{cite}

\usepackage{lscape}
\newtheorem{theorem}{\noindent Theorem}[section]
\newtheorem{proposition}[theorem]{\noindent Proposition}

\newtheorem{lemma}[theorem]{\noindent Lemma}
\newtheorem{remark}[theorem]{\noindent Remark}

\bibliography{lizh}
\UseRawInputEncoding
\begin{document}
\renewcommand{\baselinestretch}{1.3}


\begin{center}
    {\large \bf Normalized ground states for  Kirchhoff equations in ${\mathbb{R}}^{3}$\\ with a critical nonlinearity
    }
\vspace{0.5cm}\\
{Penghui Zhang$^{1,*}$,\quad Zhiqing
Han$^{1,*}$}\\\vspace{0.3cm}
{\footnotesize $^{1}$School of Mathematical Sciences, Dalian  University of Technology,
Dalian 116000, China }\\ \vspace{0.2cm}
\end{center}


\renewcommand{\theequation}{\arabic{section}.\arabic{equation}}
\numberwithin{equation}{section}


\begin{abstract}
This paper is concerned with  the existence of ground states  for a class of Kirchhoff type equation with combined power nonlinearities
\begin{equation*}
-\left(a+b\int_{\mathbb{R}^{3}}|\nabla u(x)|^{2}\right) \Delta u =\lambda u+|u|^{p-2}u+u^{5}\quad  \ \text{for some} \ \lambda\in\mathbb{R},\quad x\in\mathbb{R}^{3},
\end{equation*}
with prescribed
$L^{2}$-norm mass
\begin{equation*}
  \int_{\mathbb{R}^{3}}u^{2}=c^{2}
\end{equation*}
in Sobolev critical case and proves that the equation  has a couple of solutions $(u_{c},\lambda_{c})\in S(c)\times \mathbb{R}$ for any $c>0$, $a,b >0$ and $\frac{14}{3}\leq p< 6,$ where $S(c)=\{u\in H^{1}(\mathbb{R}^{3}):\int_{\mathbb{R}^{3}}u^{2}=c^{2}\}.$



\end{abstract}

\vspace{-1 cm}

\footnote[0]{\hspace*{-7.4mm}
$^{*}$Corresponding author:
Email address: zhangpenghui@mail.dlut.edu.cn; hanzhiq@dlut.edu.cn\\
}


\section{Introduction}
In this paper, we study the existence of normalized ground states for the following  Kirchhoff equation:
\begin{equation}\label{K}
-\left(a+b\int_{\mathbb{R}^{3}}|\nabla u(x)|^{2}\right) \Delta u =\lambda u+|u|^{p-2}u+u^{5}\quad  \ \text{for some} \ \lambda\in\mathbb{R},\quad x\in\mathbb{R}^{3},
\end{equation}
with prescribed mass
\begin{equation*}
  \int_{\mathbb{R}^{3}}u^{2}=c^{2},
\end{equation*}
where $a, b>0$ are constants and $\frac{14}{3}\leq p< 6$.
The weak solutions for the problem correspond to the critical points for the  energy functional
\begin{equation*}
  E(u)=\frac{a}{2}\int_{\mathbb{R}^{3}}|\nabla u|^{2}+\frac{b}{4}\left(\int_{\mathbb{R}^{3}}|\nabla u|^{2}\right)^{2}-\frac{1}{p}\int_{\mathbb{R}^{3}}|u|^{p}-\frac{1}{6}\int_{\mathbb{R}^{3}}|u|^{6}
\end{equation*}
on the constraint manifold
$$ S(c)=\{u\in H^{1}(\mathbb{R}^{3}):\Psi(u)=\frac{1}{2}c^{2}\},$$
where $\Psi(u)=\frac{1}{2}\int_{\mathbb{R}^{3}}u^{2}$. Equation (\ref{K}) is viewed as being nonlocal because of the appearance of the term $b\int_{\mathbb{R}^{3}}|\nabla u(x)|^{2}\Delta u$, which indicates that  equation (\ref{K}) is no longer a pointwise identity. The nonlocal term also results in  lack of  weak sequential continuity of the energy function associated to (\ref{K}), even  we remove the critical term $u^{5}$. If $\mathbb{R}^3$ is replaced by a bounded domain $\Omega\subset\mathbb{R}^{3}$, then equation (\ref{K}) describes the stationary state of the Kirchhoff type equation of the following type:
\begin{equation}\label{Kirchhoff-equ}
\begin{cases}
u_{tt} -\left(a+b\int_{\Omega}|\nabla u(x)|^{2}\right)\Delta u=f(x,u), &x\in\Omega,\\
u=0, &x\in\partial{\Omega},
\end{cases}
\end{equation}
which is presented by Kirchhoff in \cite{Kirchhoff}. It is an extension of D'Alembert's wave equation by considering the effects of the length of  strings during vibrations.

Problem (\ref{Kirchhoff-equ}) has received much attention after Lions \cite{Lions} proposed an abstract framework to deal with the problem. We refer the readers to \cite{Alves-1,Alves-2,Figueiredo-1,Figueiredo-2,Liu-Guo} and  the work \cite{Alves-2} seems to be the first one studying the critical Kirchhoff problem.

More recently, normalized solutions for elliptic equations have attracted considerable attentions,  e.g. see \cite{Bartsch-Jeanjean,Bartsch-Soave,Jeanjean-1,Liu-Guo,Luo-Zhang,Noris-1,Noris-2,Noris-3,Pierotti,Soave-1,Soave-2} and the references therein.   The work \cite{Jeanjean-1} is the first paper to deal
with the existence of normalized solutions for a second order Schr\"{o}dinger equation  with a Sobolev sub-critical and  $L^{2}$-supercritical nonlinearity and the papers \cite{Noris-1,Noris-2,Noris-3,Pierotti} deal
with the existence of normalized solutions for the problem  in bounded domains. When $b=0,$ problem (\ref{K}) becomes
\begin{equation*}
  -\Delta u =\lambda u+|u|^{p-2}u+u^{5},\quad  \ \lambda\in\mathbb{R},\quad x\in\mathbb{R}^{3},
\end{equation*}
with prescribed mass
\begin{equation*}
  \int_{\mathbb{R}^{3}}u^{2}=c^{2},
\end{equation*}
which was  recently investigated by   Soave  in  \cite{Soave-1}, where in case of  $p\in(2,6)$ the author  studied the existence and properties of the  ground states for the problem.

For  Kirchhoff type problems with a prescribed mass,  it is showed that $p=\frac{14}{3}$ is the $L^{2}$-critical exponent for the minimization problem ( \cite{Ye-1,Ye-2}). The papers \cite{Li-Ye,Ye-1,Zeng-Zhang}  consider the existence and properties  of the  $L^{2}$-subcritical constrained minimizers.  In the case of  $p\in(\frac{14}{3},6)$, the corresponding functional is unbounded from below on $S(c)$,  \cite{Luo-Wang} proved that there are infinitely many critical points by using a minimax procedure. However, few literature is concerned with normalized solutions for critical Kirchhoff problem. Inspired by \cite{Soave-1}, in this paper we
attempt to study the critical Kirchhoff problem (\ref{K}).

Our main result  is the following:
\begin{theorem}\label{main reslut}
Let $a,b >0$ and $\frac{14}{3}\leq p< 6$. Then   problem (\ref{K}) has a couple of solutions $(u_{c},\lambda_{c})\in S(c)\times \mathbb{R}$ for any $c>0$. Moreover,
\begin{equation}\label{Gound}
E(u_{c})=\inf_{u\in V(c)} E(u),
\end{equation}
where $ V(c)$ is the Pohozaev manifold defined in lemma \ref{pohozaev}.
\end{theorem}

\section{Preliminaries}
To prove our theorem, we need some notations and useful preliminary results.

Throughout this paper, we denote $B_{r}(z)$ the open ball of radius $r$ with center at $z$, and $\|\cdot\|_{p}$ the usual norm of space $L^{p}(\mathbb{R}^{3})$. Let $H=H^{1}(\mathbb{R}^{3})$ with the usual norm of space $H^{1}(\mathbb{R}^{3})$. Generic  positive constant is denoted by  $C$,  $C_{1}$,  or $C_{2}$..., which may change from line to line. Let $\mathbb{H}=H\times \mathbb{R}$ with the scalar product
\begin{equation*}
  \langle\cdot,\cdot\rangle_{\mathbb{H}}=\langle\cdot,\cdot\rangle_{{H}}+\langle\cdot,\cdot\rangle_{\mathbb{R}}
\end{equation*}
and the corresponding norm
\begin{equation*}
  \|(\cdot,\cdot)\|^{2}_{\mathbb{H}}=\|\cdot\|^{2}_{H}+|\cdot|^{2}_{\mathbb{R}}.
\end{equation*}
We denote the best  constant of  $D^{1,2}(\mathbb{R}^{3})\hookrightarrow L^{6}(\mathbb{R}^{3})$ by
\begin{equation}\label{S}
  S=\inf_{u\in D^{1,2}(\mathbb{R}^{3})\setminus\{0\}}\frac{\|\nabla u\|^{2}_{2}}{\|u\|^{2}_{6}},
\end{equation}
where $D^{1,2}(\mathbb{R}^{3})$ denotes the completion of $C_{c}^{\infty}(\mathbb{R}^{3})$ with respect to the norm $$\|u\|_{D^{1,2}(\mathbb{R}^{3})}=\|\nabla u\|_{2}.$$
In \cite{Talenti}, we know that $S$ is achieved by
\begin{equation}\label{instanton}
  U_{\varepsilon}(x)=\frac{C\varepsilon^{\frac{1}{2}}}{(\varepsilon^{2}+|x|^{2})^{\frac{1}{2}}},
\end{equation}
where $\varepsilon>0$ is a parameter.

In this paper, we will find critical points of $E$ on $H_{r}^{1}(\mathbb{R}^{3})\bigcap S(c)$, where $H_{r}^{1}(\mathbb{R}^{3})\triangleq \{u\in H^{1}(\mathbb{R}^{3}):u(x)=u(|x|)\}$ is a natural constraint.

The following Gagliardo-Nirenberg-Sobolev (GNS) inequality is  crucial in our arguments: that is,  there exists a best constant $C_{p}$ depending on $p$ such that for any $u\in H$,
\begin{equation}\label{GNS}
  \|u\|_{p}^{p}\leq C_{p}\|u\|_{2}^{(1-\beta_{p})p}\|\nabla u\|_{2}^{\beta_{p}p},
\end{equation}
where $\beta_{p}=\frac{3(p-2)}{2p}$.

As in \cite{Jeanjean-1}, we introduce the useful fiber map preserving the $L^{2}$-norm, that is,
\begin{equation*}
  H(u,s)\triangleq e^\frac{3s}{2}u(e^{s}x), \quad {\rm for\ a.e}\  x\in\mathbb{R}^{3}.
\end{equation*}

Define the auxiliary functional $I: \mathbb{H} \rightarrow \mathbb{R}$ by
\begin{equation*}
  I(u,s)=E(H(u,s))=\frac{a}{2}e^{2s}\|\nabla u\|^{2}_{2}+\frac{b}{4}e^{4s}\|\nabla u\|^{4}_{2}-\frac{e^{p\beta_{p}s}}{p}\|u\|^{p}_{p}-\frac{e^{6s}}{6}\|u\|^{6}_{6}.
\end{equation*}

The Pohozaev identity plays an important role in our discussion. We  give it in the following lemma; for  more details, see \cite{Jeanjean-1}.
\begin{lemma}\label{pohozaev}
Let $(u,\lambda)\in S(c)\times \mathbb{R}$ be a weak solution of  equation (\ref{K}). Then u belongs to the set
\begin{equation*}
  V(c)\triangleq\{u\in H: P(u)=0\}
\end{equation*}
where $$P(u)=a\|\nabla u\|^{2}_{2}+b\|\nabla u\|^{4}_{2}-\beta_{p}\|u\|^{p}_{p}-\|u\|^{6}_{6}.$$
\end{lemma}

For any $u\in S(c)$ and $s\in \mathbb{R}$, we define $\Phi_{u}(s)\triangleq I(u,s)$. Then
\begin{align*}
   (\Phi_{u})'(s)&=ae^{2s}\|\nabla u\|^{2}_{2}+be^{4s}\|\nabla u\|^{4}_{2}-\beta_{p}e^{p\beta_{p}s}\|u\|^{p}_{p}-e^{6s}\|u\|^{6}_{6}\\
   &=a\|\nabla H(u,s)\|^{2}_{2}+b\|\nabla H(u,s)\|^{4}_{2}-\beta_{p}\|H(u,s)\|^{p}_{p}-\|H(u,s)\|^{6}_{6}.
\end{align*}
Therefore, we have
\begin{lemma}\label{V(c)}
For any $u\in S(c)$, $s\in\mathbb{R}$ is a critical point for $\Phi_{u}(s)$ if and only if $H(s,u)\in V(c)$.
\end{lemma}
\begin{remark}
\begin{equation}\label{conti-H}
  {\rm The \ map }\ (u,s)\in\mathbb{H}\mapsto H(u,s)\in H \quad {\rm is \ continuous};
\end{equation}
see \cite[Lemma3.5]{Bartsch-Soave}.
\end{remark}
\begin{lemma}\cite[Lemma 3.3]{Liu-Guo}\label{system}
For $t, s>0$, the following system
\begin{equation*}
  \begin{cases}
   x(t,s)=t-aS(t+s)^{\frac{1}{3}}=0\\
   y(t,s)=s-bS^{2}(t+s)^{\frac{2}{3}}=0
  \end{cases}
\end{equation*}
has a unique solution $(t_{0},s_{0})$. Moreover, if $x(t,s)\geq0$ and $y(t,s)\geq0$, then $t\geq t_{0}$, $t\geq t_{0}$.
\end{lemma}

\section{Characterization of mountain pass level}
As in \cite {Jeanjean-1, Luo-Zhang}, we firstly prove that $I(u,s)$ has the mountain pass geometry on $S(c)\times \mathbb{R}$ in the following lemmas.
\begin{lemma}\label{3.1}
Assume that $a, b>0$ and $\frac{14}{3}\leq p< 6$. Let $u\in S(c)$ be arbitrary fixed. Then
\begin{itemize}
  \item [(1)] $\int_{\mathbb{R}^{3}}|\nabla H(u,s)|^{2}\rightarrow 0,$ and $I(u,s)\rightarrow 0^{+},$  as $s\rightarrow -\infty ;$
  \item [(2)] $\int_{\mathbb{R}^{3}}|\nabla H(u,s)|^{2}\rightarrow +\infty,$ and $I(u,s)\rightarrow -\infty,$ as $s\rightarrow +\infty .$
\end{itemize}
\end{lemma}
\begin{proof}
The proof  is trivial from the facts
\begin{equation*}
  \|\nabla H(s,u)\|^{2}_{2}=e^{2s}\|\nabla u\|^{2}_{2}
\end{equation*}
and
\begin{equation*}
  I(u,s)=E(H(u,s))=\frac{a}{2}e^{2s}\|\nabla u\|^{2}_{2}+\frac{b}{4}e^{4s}\|\nabla u\|^{4}_{2}-\frac{e^{p\beta_{p}s}}{p}\|u\|^{p}_{p}-\frac{e^{6s}}{6}\|u\|^{6}_{6}.
\end{equation*}
\end{proof}

\begin{lemma}\label{3.2}
Let $a, b,c>0$ and $\frac{14}{3}\leq p< 6$. Then there exists $K_{c}>0$ such that
\begin{equation*}
 P(u),E(u)> 0 \ {\rm for \ all \ } u\in A_{c}, \quad {\rm and}\quad 0<\sup_{u\in A_{c}}E(u)<\inf_{u\in B_{c}}E(u)
\end{equation*}
with
\begin{equation*}
  A_{c}=\{u\in S_{c}:\int_{\mathbb{R}^{3}}|\nabla u|^{2}\leq K_{c}\}, \quad B_{c}=\{u\in S_{c}:\int_{\mathbb{R}^{3}}|\nabla u|^{2}=2K_{c}\}.
\end{equation*}
\end{lemma}
\begin{proof}
Let $K>0$ be arbitrary fixed and suppose that $u,v\in S(c)$ are such that $\|\nabla u\|_{2}^{2}\leq K$ and $\|\nabla v\|_{2}^{2}=2K.$ Then, for $K>0$ small enough, using (\ref{GNS}) and $p\beta_{p}\geq4$, there exist two constants $C_{1}$
and $C_{2}$ such that
\begin{align*}
 P(u) & \geq a\|\nabla u\|^{2}_{2}+b\|\nabla u\|^{4}_{2}-C_{1}\|\nabla u\|^{p\beta_{p}}_{2}-C_{2}\|\nabla u\|^{6}_{2},
\end{align*}
\begin{align*}
 E(u) & \geq \frac{a}{2}\|\nabla u\|^{2}_{2}+\frac{b}{4}\|\nabla u\|^{4}_{2}-C_{1}\|\nabla u\|^{p\beta_{p}}_{2}-C_{2}\|\nabla u\|^{6}_{2}
\end{align*}
and
\begin{align*}
  E(v)-E(u)&\geq E(v)-\frac{a}{2}\|\nabla u\|^{2}_{2}-\frac{b}{4}\|\nabla u\|^{4}_{2}\\
           &\geq \frac{aK}{2}+\frac{bK^{2}}{2}-C_{1}\|\nabla v\|^{p\beta_{p}}_{2}-C_{2}\|\nabla v\|^{6}_{2}\\
           &\geq \frac{aK}{2}+\frac{bK^{2}}{2}-C_{1}K^{\frac{p\gamma_{p}}{2}}-C_{2}K^{3}.
\end{align*}
Therefore, by the above inequalities, it follows that there exists $K_{c}$ small enough such that
\begin{equation*}
P(u),  E(u)> 0 \ {\rm for \ all \ } x\in A_{c}, \quad {\rm and}\quad 0<\sup_{u\in A_{c}}E(u)<\inf_{u\in B_{c}}E(u)
\end{equation*}
with
\begin{equation*}
  A_{c}=\{u\in S_{c}:\int_{\mathbb{R}^{3}}|\nabla u|^{2}\leq K_{c}\}, \quad B_{c}=\{u\in S_{c}:\int_{\mathbb{R}^{3}}|\nabla u|^{2}=2K_{c}\}.
\end{equation*}
\end{proof}

Next, we give a characterization of mountain pass level for $I(u,s)$ and $E(u)$. $E^{d}$  denotes  the  set $\{ u\in S_{c}: E(u)\leq d\} $.
\begin{proposition}\label{pro3.3}
Under assumptions that $a, b>0$ and $\frac{14}{3}\leq p< 6$, let
$$\widetilde{\gamma}_{c}=\inf_{\widetilde{h}\in\widetilde{\Gamma_{c}}}\max_{t\in[0,1]}I(\widetilde{h}(t))$$
where
\begin{equation*}
  \widetilde{\Gamma_{c}}=\{\widetilde{h}\in C([0,1],S(c)\times \mathbb{R}):\widetilde{h}(0)\in(A_{c},0), \ \widetilde{h}(1)\in(E^{0},0)\},
\end{equation*}
and $$\gamma_{c}=\inf_{h\in{\Gamma_{c}}}\max_{t\in[0,1]}E(h(t))$$
where
\begin{equation*}
  {\Gamma_{c}}=\{{h}\in C([0,1],S(c)):h(0)\in A_{c},\ h(1)\in E^{0}\}.
\end{equation*}
Then we have
$$\widetilde{\gamma}_{c} ={\gamma}_{c}.$$
\end{proposition}
\begin{proof}
 Since $\Gamma_{c}\times\{0\}\subseteq\widetilde{\Gamma_{c}},$  we have $\widetilde{\gamma}_{c}\leq{\gamma_{c}}$. So, it remains to prove that $\widetilde{\gamma}_{c}\geq{\gamma_{c}}$. For any  $\widetilde{h}(t)=(\widetilde{h_{1}}(t),\widetilde{h_{2}}(t))\in\widetilde{\Gamma}_{c}$, we set $h(t)=H(\widetilde{h_{1}}(t),\widetilde{h_{2}}(t))$. Then $h(t)\in\Gamma_{c}$ and
\begin{equation*}
  \max_{t\in[0,1]}I(\widetilde{h}(t))=\max_{t\in[0,1]}E(H(\widetilde{h_{1}}(t),\widetilde{h_{2}}(t))
  =\max_{t\in[0,1]}E(h(t)),
\end{equation*}
which shows that $\widetilde{\gamma}_{c}\geq{\gamma_{c}}$.
\end{proof}
In the following proposition, we give the  existence of $(PS)_{\widetilde{\gamma}_{c}}$ sequence for $I(u,s)$. Its proof is by a standard argument using the Ekeland's Variational principle and constructing pseudo-gradient flow.
\begin{proposition}\cite[Proposition2.2]{Jeanjean-1}\label{pro3.4}
Let $\{g_{n}\}\subset \widetilde{\Gamma}_{c}$ be such that
\begin{equation*}
  \max_{t\in[0,1]}I(g_{n}(t))\leq\widetilde{\gamma}_{c}+\frac{1}{n}.
\end{equation*}
Then there exists a sequence $\{(u_{n},s_{n})\}\subset S(c)\times\mathbb{R}$ such that
\begin{itemize}
  \item [(1)] $I(u_{n},s_{n}) \in [{\gamma}_{c}-\frac{1}{n},\gamma_{c}+\frac{1}{n}]$;
  \item [(2)] $\min_{t\in[0,1]} \| (u_{n},s_{n})-g_{n}(t)\|_{\mathbb{H}}\leq\frac{1}{\sqrt n}$;
  \item [(3)] $\|I'|_{S(c)\times\mathbb{R}}(u_{n},s_{n})\|\leq\frac{2}{\sqrt n}$ i.e.
  \begin{equation*}
    |\langle I'(u_{n},s_{n}),z\rangle_{\mathbb{H}^{{-1}}\times \mathbb{H}}|\leq\frac{2}{\sqrt n}\|z\|_\mathbb{H}
  \end{equation*}
  for all
  \begin{equation*}
    z\in\widetilde{T}_{(u_{n},s_{n})}\triangleq\{(z_{1},z_{2})\in \mathbb{H}:\langle u_{n},z_{1}\rangle_{L^{2}}=0\}.
  \end{equation*}

\end{itemize}
\end{proposition}
\begin{proposition}\label{PPS}
Under the assumptions $a, b>0$ and $\frac{14}{3}\leq p< 6$,  there exists a sequence $\{v_{n}\}\subset S(c)$ such that
\begin{enumerate}[(1)]
      \item  $E(v_{n})\rightarrow \gamma_{c}$, as $n\rightarrow\infty$;
      \item  $P(v_{n})\rightarrow 0$, as $n\rightarrow\infty$;
      \item  $E'|_{S(c)}(v_{n})\rightarrow 0$, as $n\rightarrow\infty$  i.e.
  \begin{equation*}
    |\langle E'(v_{n}),h\rangle_{{H}^{-1}\times H}|\rightarrow 0
  \end{equation*}
uniformly for all  $h$ satisfying
  \begin{equation*}
   \|h\|_{H}\leq 1 \quad {\rm where}\quad h\in{T}_{v_{n}}\triangleq\{h\in {H}:\langle v_{n},h\rangle_{L^{2}}=0\}.
  \end{equation*}
\end{enumerate}
\end{proposition}
\begin{proof}
By Proposition \ref{pro3.3}, $\widetilde{\gamma}_{c}= {\gamma_{c}}$. Pick $\{g_{n}=((g_{n})_{1},0)\}\subset \widetilde{\Gamma}_{c}$ such that
\begin{equation*}
  \max_{t\in[0,1]}I(g_{n}(t))\leq\widetilde{\gamma}_{c}+\frac{1}{n}.
\end{equation*}
It follows from Proposition \ref{pro3.4} that there exists a sequence $\{(u_{n},s_{n})\}\subset S(c)\times\mathbb{R}$ such that, as $n\rightarrow\infty$, one has
\begin{align}
  &I(u_{n},s_{n})\rightarrow {\gamma}_{c},\label{gamma}\\
   &s_{n}\rightarrow 0,\label{sn} \\
   &\partial_{s}I(u_{n},s_{n})\rightarrow 0.\label{partial-s}
\end{align}
Let $v_{n}=H(u_{n},s_{n})$.  Then $E(v_{n})=I(u_{n},s_{n})$  and, by (\ref{gamma}),  $(1)$ holds. For the proof of  $(2)$, we notice that
\begin{align*}
  \partial_{s}I(u_{n},s_{n})&= ae^{2s_{n}}\|\nabla u_{n}\|^{2}_{2}+be^{4s_{n}}\|\nabla u_{n}\|^{4}_{2}-e^{p\beta_{p}s_{n}}\|u_{n}\|^{p}_{p}-e^{6s_{n}}\|u_{n}\|^{6}_{6}\\
  &=a\|\nabla v_{n}\|^{2}_{2}+b\|\nabla v_{n}\|^{4}_{2}-\beta_{p}\|v_{n}\|^{p}_{p}-\|v_{n}\|^{6}_{6}\\
  &=P(v_{n}),
\end{align*}
which implies $(2)$ by (\ref{partial-s}).

For the proof of  $(3)$, let $h_{n}\in T_{v_{n}}$.  We have
\begin{align*}
  \langle E'(v_{n}),h_{n}\rangle_{H^{-1}\times H}&=a\int_{\mathbb{R}^{3}}\nabla v_{n}(x)\nabla h_{n}(x)+b\int_{\mathbb{R}^{3}}|\nabla v_{n}(x)|^{2}\int_{\mathbb{R}^{3}}\nabla v_{n}(x)\nabla h_{n}(x)\\
  &-\int_{\mathbb{R}^{3}}|v_{n}(x)|^{p-{2}}v_{n}(x)h_{n}(x)-\int_{\mathbb{R}^{3}}(v_{n}(x))^{5}h_{n}(x)\\
  &=ae^{\frac{5s_{n}}{2}}\int_{\mathbb{R}^{3}}\nabla u_{n}(e^{s_{n}}x)\nabla h_{n}(x)-e^{\frac{15s_{n}}{2}}\int_{\mathbb{R}^{3}}(u_{n}(e^{s_{n}}x))^{5}h_{n}(x)\\
  &+be^{\frac{15s_{n}}{2}}\int_{\mathbb{R}^{3}}|\nabla u_{n}(e^{s_{n}}x)|^{2}\int_{\mathbb{R}^{3}}\nabla u_{n}(e^{s_{n}}x)\nabla h_{n}(x)\\
  &-e^{\frac{3(p-1)s_{n}}{2}}\int_{\mathbb{R}^{3}}|u_{n}(e^{s_{n}}x)|^{p-{2}}u_{n}(e^{s_{n}}x)h_{n}(x)\\
  &=ae^{2s_{n}}\int_{\mathbb{R}^{3}}\nabla u_{n}(x)e^{\frac{-5s_{n}}{2}}\nabla h_{n}(e^{-s_{n}}x)\\
  &+be^{4s_{n}}\int_{\mathbb{R}^{3}}|\nabla u_{n}(x)|^{2}\int_{\mathbb{R}^{3}}\nabla u_{n}(x)e^{\frac{-5s_{n}}{2}}\nabla h_{n}(e^{-s_{n}}x)\\
  &-e^{p\beta_{p}s_{n}}\int_{\mathbb{R}^{3}}|u_{n}(x)|^{p-{2}}u_{n}(x)
  e^{\frac{-3s_{n}}{2}}h_{n}(e^{-s_{n}}x)\\
  &-e^{6s_{n}}\int_{\mathbb{R}^{3}}(u_{n}(x))^{5}e^{\frac{-3s_{n}}{2}}h_{n}(e^{-s_{n}}x).
\end{align*}
Setting $\widehat{h}_{n}(x)=e^{\frac{-3s_{n}}{2}}h_{n}(e^{-s_{n}}x),$ then
\begin{equation*}
  \langle I'(u_{n},s_{n}),(\widehat{h}_{n},0)\rangle_{\mathbb{H}^{{-1}}\times \mathbb{H}}=\langle E'(v_{n}),h_{n}\rangle_{H^{-1}\times H}.
\end{equation*}
It is easy to see that
\begin{align*}
  \langle u_{n}(x),\widehat{h}_{n}(x)\rangle_{L^{2}}&=\int_{\mathbb{R}^{3}}u_{n}(x)
  e^{\frac{-3s_{n}}{2}}h_{n}(e^{-s_{n}}x)\\
  &=\int_{\mathbb{R}^{3}}u_{n}(e^{s_{n}}x)
  e^{\frac{3s_{n}}{2}}h_{n}(x)\\
  &=\int_{\mathbb{R}^{3}}v_{n}(x)h_{n}(x)=0.
\end{align*}
So, we have that $(\widehat{h}_{n}(x),0)\in\widetilde{T}_{(u_{n},s_{n})}$. On the other hand,
\begin{align*}
  \|(\widehat{h}_{n}(x),0)\|^{2}_{\mathbb{H}}&= \|\widehat{h}_{n}(x)\|^{2}_{H} \\
  &=\|h_{n}(x)\|^{2}_{2}+e^{-2s_{n}}\|\nabla h_{n}(x)\|^{2}_{2}\\
  &\leq C\|\widehat{h}_{n}(x)\|^{2}_{H},
\end{align*}
where the last inequality holds by (\ref{partial-s}). Thus,  $(3)$ is proved.
\end{proof}
In the following lemma, we  give an upper bound estimate for the mountain pass level $\gamma_{c}$.
\begin{lemma}\label{m-p-l-upbdd}
Under assumptions $a, b>0$ and $\frac{14}{3}\leq p< 6$, then $\gamma_{c}<\gamma^{*}_{c}\triangleq \frac{abS^{3}}{4}+\frac{b^{3}S^{6}}{24}+\frac{(4as+b^{2}S^{4})^{\frac{3}{2}}}{24}$, where $S$ is defined in $(\ref{S}$).
\end{lemma}
\begin{proof}
Let $\varphi(x)\in C^{\infty}_{c}(B_{2}(0))$ be a radial cut-off function such that $0\leq\varphi(x)\leq1$ and $\varphi(x)\equiv1$ on $B_{1}(0)$. Then we take $u_{\varepsilon}=\varphi(x)U_{\varepsilon}$ ($U_{\varepsilon}$ defined in (\ref{instanton})) and $$v_{\varepsilon}=c\frac{u_{\varepsilon}}{\|u_{\varepsilon}\|_{2}}\in S(c)\cap H_{r}^{1}.$$
We take $\varepsilon=1$ and define
\begin{equation*}
  K_{1}\triangleq \|\nabla U_{1}\|^{2}_{2}, \quad  K_{2}\triangleq \| U_{1}\|^{6}_{6},\quad K_{3}\triangleq \| U_{1}\|^{p}_{p}.
\end{equation*}
According to \cite[Appendix A]{Soave-1}, we have
\begin{equation}\label{estimate-instanton}
  \begin{cases}
K_{1}/ K_{2}=S,\quad \|\nabla u_{\varepsilon}\|^{2}_{2}=K_{1}+O(\varepsilon),\quad\|u_{\varepsilon}\|^{2}_{6}=K_{2}+O(\varepsilon^{2}),\\
\|u_{\varepsilon}\|^{p}_{p}=\varepsilon^{3-\frac{p}{2}}\left(K_{3}+O(\varepsilon^{p-3})\right),\quad \|u_{\varepsilon}\|^{2}_{2}=O(\varepsilon^{2})+\omega\left(\int^{2}_{0}\varphi(r)dr\right)\varepsilon,
  \end{cases}
\end{equation}
where $\omega$ is the area of  the unit sphere in $\mathbb{R}^{3}$. Define
$$\Psi_{v_{\varepsilon}}(s)\triangleq\frac{a}{2}e^{2s}\|\nabla v_{\varepsilon}\|^{2}_{2}+\frac{b}{4}e^{4s}\|\nabla v_{\varepsilon}\|^{4}_{2}-\frac{e^{6s}}{6}\|v_{\varepsilon}\|^{6}_{6}.$$
Then
$$(\Psi)'_{v_{\varepsilon}}(s)=ae^{2s}\|\nabla v_{\varepsilon}\|^{2}_{2}+be^{4s}\|\nabla v_{\varepsilon}\|^{4}_{2}-e^{6s}\|v_{\varepsilon}\|^{6}_{6}.$$
\textbf{Step 1:} It is easy to see that $\Psi_{v_{\varepsilon}}(s)$ has a unique critical point $s_{0},$ which is a strict maximum point such that
\begin{equation}\label{maximum point}
  e^{2s_{0}}=\frac{b\|\nabla v_{\varepsilon}\|_{2}^{4}+\sqrt{b^{2}\|\nabla v_{\varepsilon}\|_{2}^{8}+4a\|\nabla v_{\varepsilon}\|_{2}^{2}\| v_{\varepsilon}\|_{6}^{6}}}{2\| v_{\varepsilon}\|^{6}_{6}}
\end{equation}
and the maximum level of $\Psi_{v_{\varepsilon}}(s)$ is
\begin{align}\label{max-esti}
\Psi_{v_{\varepsilon}}(s_{0})&=\frac{a}{2}\|\nabla v_{\varepsilon}\|^{2}_{2}\left(\frac{b\|\nabla v_{\varepsilon}\|_{2}^{4}+\sqrt{b^{2}\|\nabla v_{\varepsilon}\|_{2}^{8}+4a\|\nabla v_{\varepsilon}\|_{2}^{2}\| v_{\varepsilon}\|_{6}^{6}}}{2\| v_{\varepsilon}\|^{6}_{6}}\right)\nonumber\\
&+\frac{b}{4}\|\nabla v_{\varepsilon}\|^{4}_{2}{\left(\frac{b\|\nabla v_{\varepsilon}\|_{2}^{4}+\sqrt{b^{2}\|\nabla v_{\varepsilon}\|_{2}^{8}+4a\|\nabla v_{\varepsilon}\|_{2}^{2}\| v_{\varepsilon}\|_{6}^{6}}}{2\| v_{\varepsilon}\|^{6}_{6}}\right)}^{2}\nonumber\\
&-\frac{1}{6}\| v_{\varepsilon}\|^{6}_{6}{\left(\frac{b\|\nabla v_{\varepsilon}\|_{2}^{4}+\sqrt{b^{2}\|\nabla v_{\varepsilon}\|_{2}^{8}+4a\|\nabla v_{\varepsilon}\|_{2}^{2}\| v_{\varepsilon}\|_{6}^{6}}}{2\| v_{\varepsilon}\|^{6}_{6}}\right)}^{3}\nonumber\\
&=\frac{ab\|\nabla v_{\varepsilon}\|^{6}_{2}}{4\|v_{\varepsilon}\|^{6}_{6}}+\frac{b^{3}\|\nabla v_{\varepsilon}\|^{12}_{2}}{24\|v_{\varepsilon}\|^{12}_{6}}+\frac{\left({b^{2}\|\nabla v_{\varepsilon}\|_{2}^{8}+4a\|\nabla v_{\varepsilon}\|_{2}^{2}\| v_{\varepsilon}\|_{6}^{6}}\right)^{\frac{3}{2}}}{24\|v_{\varepsilon}\|^{12}_{6}}.
\end{align}
By (\ref{estimate-instanton}), we conclude that
\begin{align*}
  &\frac{ab\|\nabla v_{\varepsilon}\|^{6}_{2}}{4\|v_{\varepsilon}\|^{6}_{6}}=\frac{ab}{4}\frac{\|\nabla u_{\varepsilon}\|^{6}_{2}}{\|u_{\varepsilon}\|^{6}_{6}}=\frac{ab}{4}\left(\frac{K_{1}+O(\varepsilon)}{K_{2}+
  O(\varepsilon^{2})}\right)^{3}=\frac{ab}{4}S^{3}+O(\varepsilon);\\
  &\frac{b^{3}\|\nabla v_{\varepsilon}\|^{12}_{2}}{24\|v_{\varepsilon}\|^{12}_{6}}=\frac{b^{3}}{24}\frac{\|\nabla u_{\varepsilon}\|^{12}_{2}}{\|u_{\varepsilon}\|^{12}_{6}}=\frac{b^{3}}{24}\left(\frac{K_{1}+O(\varepsilon)}{K_{2}+
  O(\varepsilon^{2})}\right)^{6}=\frac{b^{3}}{24}S^{6}+O(\varepsilon).\\
  &
\end{align*}
For the last term in (\ref{max-esti})
\begin{align*}
  \frac{\left({b^{2}\|\nabla v_{\varepsilon}\|_{2}^{8}+4a\|\nabla v_{\varepsilon}\|_{2}^{2}\| v_{\varepsilon}\|_{6}^{6}}\right)^{\frac{3}{2}}}{24\|v_{\varepsilon}\|^{12}_{6}}&=\frac{1}{24}
  \left((b^{2}\frac{K_{1}}{K_{2}}+O(\varepsilon))^{4} +4a(\frac{K_{1}}{K_{2}}+O(\varepsilon))\right)^{\frac{3}{2}} \\
   & =\frac{1}{24}\left(4aS +b^{2}S^{4}+O(\varepsilon)\right)^{\frac{3}{2}}\\
   & =\frac{1}{24}\left(4aS +b^{2}S^{4}\right)^{\frac{3}{2}}+O(\varepsilon).
\end{align*}
By the above estimates, one has
\begin{equation}\label{m-p-esti}
  \Psi_{v_{\varepsilon}}(s_{0})=\frac{abS^{3}}{4}+\frac{b^{3}S^{6}}{24}+\frac{\left(4aS +b^{2}S^{4}\right)^{\frac{3}{2}}}{24}+O(\varepsilon).
\end{equation}
\textbf{Step 2:} We give an upper bound estimate for $\Phi_{v_{\varepsilon}}(s)=I(v_{\varepsilon},s)$. Note that
\begin{equation*}
   (\Phi_{v_{\varepsilon}})'(s)=ae^{2s}\|\nabla v_{\varepsilon}\|^{2}_{2}+be^{4s}\|\nabla v_{\varepsilon}\|^{4}_{2}-\beta_{p}e^{p\beta_{p}s}\|v_{\varepsilon}\|^{p}_{p}
   -e^{6s}\|v_{\varepsilon}\|^{6}_{6}.
\end{equation*}
Obviously, $\Phi_{v_{\varepsilon}}(s)$ has a unique critical point $s_{1}$ and
\begin{equation*}
  e^{2s_{1}}\leq e^{2s_{0}}=\frac{b\|\nabla v_{\varepsilon}\|_{2}^{4}+\sqrt{b^{2}\|\nabla v_{\varepsilon}\|_{2}^{8}+4a\|\nabla v_{\varepsilon}\|_{2}^{2}\| v_{\varepsilon}\|_{6}^{6}}}{2\| v_{\varepsilon}\|^{6}_{6}}.
\end{equation*}
Since
\begin{equation*}
  (\Phi_{v_{\varepsilon}})'(s_{1})=e^{2s_{1}}\|\nabla v_{\varepsilon}\|^{2}_{2}+e^{4s_{1}}\|\nabla v_{\varepsilon}\|^{4}_{2}-\beta_{p}e^{p\beta_{p}s_{1}}\|v_{\varepsilon}\|^{p}_{p}
   -e^{6s_{1}}\|v_{\varepsilon}\|^{6}_{6}=0,
\end{equation*}
in view of  the definition of $\beta_{p}$ with $p\beta_{p}\geq4$ and (\ref{estimate-instanton}),  we have
\begin{align}\label{s-1-esti}
e^{2s_{1}}&= \frac{b|\nabla v_{\varepsilon}\|_{2}^{4}}{\|v_{\varepsilon}\|^{6}_{6}}+e^{-2s_{1}}\frac{a\|\nabla v_{\varepsilon}\|^{2}}{\|v_{\varepsilon}\|^{6}_{6}}-\beta_{p} e^{2s_{1}{\frac{(p\beta_{p}-4)}{2}}}\frac{\|v_{\varepsilon}\|^{p}_{p}}{\|v_{\varepsilon}\|^{6}_{6}}\nonumber\\
&\geq\frac{b|\nabla v_{\varepsilon}\|_{2}^{4}}{\|v_{\varepsilon}\|^{6}_{6}}+e^{-2s_{0}}\frac{a\|\nabla v_{\varepsilon}\|^{2}}{\|v_{\varepsilon}\|^{6}_{6}}-\beta_{p} e^{2s_{0}{\frac{(p\beta_{p}-4)}{2}}}\frac{\|v_{\varepsilon}\|^{p}_{p}}{\|v_{\varepsilon}\|^{6}_{6}}\nonumber\\
&=\frac{b}{c^{2}}\frac{\|\nabla u_{\varepsilon}\|^{4}_{2}}{\| u_{\varepsilon}\|^{4}_{6}}\cdot\frac{\| u_{\varepsilon}\|^{2}_{2}}{\| u_{\varepsilon}\|^{2}_{6}}\nonumber\\
&+\frac{a\|\nabla v_{\varepsilon}\|^{2}_{2}}{\|v_{\varepsilon}\|^{6}_{6}}\cdot \frac{2\| v_{\varepsilon}\|^{6}_{6}}{b\|\nabla v_{\varepsilon}\|_{2}^{4}+\sqrt{b^{2}\|\nabla v_{\varepsilon}\|_{2}^{8}+4a\|\nabla v_{\varepsilon}\|_{2}^{2}\| v_{\varepsilon}\|_{6}^{6}}}\nonumber\\
&-\frac{\gamma_{p}\| v_{\varepsilon}\|^{p}_{p}}{\|v_{\varepsilon}\|^{6}_{6}}\cdot\left(\frac{b\|\nabla v_{\varepsilon}\|_{2}^{4}+\sqrt{b^{2}\|\nabla v_{\varepsilon}\|_{2}^{8}+4a\|\nabla v_{\varepsilon}\|_{2}^{2}\| v_{\varepsilon}\|_{6}^{6}}}{2\| v_{\varepsilon}\|^{6}_{6}}\right)^{\frac{(p\beta_{p}-4)}{2}}\nonumber\\
&=C_{1}\varepsilon+O(\varepsilon^{2})+C_{2}\varepsilon+O(\varepsilon^{2}) -C_{3}(\varepsilon+O(\varepsilon^{2}))^{\frac{11}{2}-\frac{3p}{4}}.
\end{align}
In view of
\begin{align*}
\Phi_{v_{\varepsilon}}(s_{1})&=\Psi_{v_{\varepsilon}}(s_{1})-
\frac{e^{p\beta_{p}s_{1}}}{p}\|v_{\varepsilon}\|^{p}_{p}\\
&\leq\Psi_{v_{\varepsilon}}(s_{0})-
\frac{e^{p\beta_{p}s_{1}}}{p}\|v_{\varepsilon}\|^{p}_{p}
\end{align*}
and
$$\|v_{\varepsilon}\|^{p}_{p}=C_{4}\varepsilon^{3-p}+O(1),$$
letting $\varepsilon\rightarrow 0$, we have
\begin{align*}
         \Phi_{v_{\varepsilon}}(s_{1})&\leq \frac{abS^{3}}{4}+\frac{b^{3}S^{6}}{24}+\frac{\left(4aS +b^{2}S^{4}\right)^{\frac{3}{2}}}{24}+O(\varepsilon)-
         \frac{e^{p\beta_{p}s_{1}}}{p}\|v_{\varepsilon}\|^{p}_{p}\\
         &\leq \frac{abS^{3}}{4}+\frac{b^{3}S^{6}}{24}+\frac{\left(4aS +b^{2}S^{4}\right)^{\frac{3}{2}}}{24}+O(\varepsilon)-(e^{2s_{1}})^{\frac{p\beta_{p}}{2}}
         \frac{\|v_{\varepsilon}\|^{p}_{p}}{p}\\
         &\leq \frac{abS^{3}}{4}+\frac{b^{3}S^{6}}{24}+\frac{\left(4aS +b^{2}S^{4}\right)^{\frac{3}{2}}}{24}+O(\varepsilon)-C_{5}\varepsilon^{\frac{6-p}{4}}\\
         &<\frac{abS^{3}}{4}+\frac{b^{3}S^{6}}{24}+\frac{\left(4aS +b^{2}S^{4}\right)^{\frac{3}{2}}}{24},
\end{align*}
 by (\ref{m-p-esti}) and (\ref{s-1-esti}).

\noindent\textbf{Step 3:} Take $\varepsilon$ small enough such that $v_{\varepsilon}$ satisfies the above inequality. By Lemma \ref{3.1} and (\ref{conti-H}), there exist $s^{-}\ll -1$ and $s^{+}\gg 1$ such that
\begin{equation}\label{path}
  h_{v_{\varepsilon}}:\tau\in[0,1]\mapsto H(v_{\varepsilon},(1-\tau)s^{-}+\tau s^{+})\in \Gamma.
\end{equation}
Therefore,
\begin{equation*}
  \gamma_{c}\leq \max_{\tau\in[0,1]}E( h_{v_{\varepsilon}}(\tau))\leq\Phi_{v_{\varepsilon}}(s_{1})
  <\frac{abS^{3}}{4}+\frac{b^{3}S^{6}}{24}+\frac{\left(4aS +b^{2}S^{4}\right)^{\frac{3}{2}}}{24}.
\end{equation*}
By letting  $\gamma^{*}_{c}\triangleq \frac{abS^{3}}{4}+\frac{b^{3}S^{6}}{24}+\frac{(4as+b^{2}S^{4})^{\frac{3}{2}}}{24}$, the conclusion follows from the above inequality.
\end{proof}
Let $m(c)\triangleq \inf_{u\in V(c)} E(u)$, where $ V(c)$ is the Pohozaev manifold.  We have    the following relationship between $\gamma_{c}$ and $m(c)$.
\begin{lemma}\label{gammac-mc}
Under the assumptions  $a, b>0$ and $\frac{14}{3}\leq p< 6$, we have  that
\begin{equation*}
  m(c)=\gamma_{c}>0.
\end{equation*}
\end{lemma}
\begin{proof}
\textbf{Step 1:} We claim that for every $u\in S(c)$, there exists a unique $t_{u}\in \mathbb{R}$ such that $H(u,t_{u})\in V(c)$, where $t_{u}$ is a strict  maximum point for $\Phi_{u}(s)$ at a positive level.

The  existence of  $t_{u}$ follows from Lemma \ref{V(c)}. The uniqueness is  from the following reasoning.  Noticing that
\begin{align*}
   (\Phi_{u})''(s)&=2ae^{2s}\|\nabla u\|^{2}_{2}+4be^{4s}\|\nabla u\|^{4}_{2}-p{\beta_{p}}^{2}e^{p\beta_{p}s}\|u\|^{p}_{p}-6e^{6s}\|u\|^{6}_{6}\\
   &=2a\|\nabla H(u,s)\|^{2}_{2}+4b\|\nabla H(u,s)\|^{4}_{2}-p{\beta_{p}}^{2}\|H(u,s)\|^{p}_{p}-6\|H(u,s)\|^{6}_{6},
\end{align*}
combining with $(\Phi_{u})'(t_{u})=0$, we have
\begin{equation*}
  (\Phi_{u})''(t_{u})=-2a\|\nabla H(u,t_{u})\|^{2}_{2}-\beta_{p}(p{\beta_{p}}-4)\|H(u,t_{u})\|^{p}_{p}-2\|H(u,t_{u})\|^{6}_{6}<0.
\end{equation*}
By Lemma \ref{3.1},  $\Phi_{u}(s)$ has a global maximum point at a positive level.

\noindent\textbf{Step 2:} We claim that $E(u)\leq 0$ implies $P(u)<0$.

Since $(\Phi_{u})''(t_{u})<0$, we know that $\Phi_{u}$ is strictly decreasing and concave on $(t_{u},+\infty)$. By the claim of Step 1,
\begin{equation*}
E(H(u,s))\leq E(H(u,0))=E(u)\leq 0 \quad  {\rm for}\ s\geq0.
\end{equation*}
So, we have $t_{u}<0$. Since
\begin{equation*}
P(H(u,t_{u}))= (\Phi_{u})'(t_{u})=0 \quad {\rm and } \quad P(u)=P(H(u,0))= (\Phi_{u})'(0),
\end{equation*}
we obtain that $P(u)< 0.$

\noindent\textbf{Step 3:} $\gamma_{c}=m(c)$.

Let $u\in S(c)$.  As in (\ref{path}), we define a path
\begin{equation*}
  h_{u}:\tau\in[0,1]\mapsto H(u,(1-\tau)s^{-}+\tau s^{+})\in \Gamma.
\end{equation*}
By Step 1, we have
\begin{equation*}
  \max_{\tau\in[0,1]}E( h_{u}(\tau))\geq \gamma_{c}.
\end{equation*}
  So we have $\gamma_{c}\leq m(c)$. On the other hand, for any  $\widetilde{h}(\tau)=(\widetilde{h_{1}}(\tau),\widetilde{h_{2}}(\tau))\in\widetilde{\Gamma}_{c}$, we consider the function
\begin{equation*}
  \widetilde{P}(\tau)=P(H(\widetilde{h_{1}}(\tau),\widetilde{h_{2}}(\tau)))\in\mathbb{R}.
\end{equation*}
Since $H(\widetilde{h_{1}}(0),\widetilde{h_{2}}(0))=\widetilde{h_{1}}(0)\in A_{c}$ and $H(\widetilde{h_{1}}(1),\widetilde{h_{2}}(1))=\widetilde{h_{1}}(1)\in E^{0},$ hence by Lemma \ref{3.2}, we deduce that
\begin{equation*}
 \widetilde{P}(0)=\widetilde{P}( \widetilde{h_{1}}(0))>0,
\end{equation*}
and using the claim of Step 2,
\begin{equation*}
  \widetilde{P}(1)=\widetilde{P}( \widetilde{h_{1}}(1))<0.
\end{equation*}
The function  $\widetilde{P}(\tau)$ is continuous  by (\ref{conti-H}), and hence we deduce that there exists $\bar{\tau}\in(0,1)$ such that $\widetilde{P}(\bar{\tau})=0$, which  implies that $H(\widetilde{h_{1}}(\bar{\tau}),\widetilde{h_{2}}(\bar{\tau}))\in V(c)$ and
\begin{equation*}
  \max_{\tau\in[0,1]}I(\widetilde{h}(\tau))=\max_{\tau\in[0,1]}
  E(H(\widetilde{h_{1}}(\tau),\widetilde{h_{2}}(\tau)))\geq \inf_{u\in V(c)}E(u).
\end{equation*}
So, we infer that $\widetilde{\gamma}_{c} ={\gamma}_{c}\geq m(c)$.

\noindent\textbf{Step 4:} At last, we will prove that $\gamma_{c}>0$.

If $u\in V(c)$, then $P(u)=0$  and by GNS inequality (\ref{GNS}), we deduce that
\begin{equation*}
a\|\nabla u\|^{2}_{2}+b\|\nabla u\|^{4}_{2} \leq C\|\nabla u\|_{2}^{\beta_{p}p}+C\|\nabla u\|_{2}^{6},
\end{equation*}
 which implies that there exists $\delta>0$ such that $\inf_{V(c)}\|\nabla u\|_{2}\geq \delta$. For any $u\in V(c)$, in view of $p\beta_{p}\geq 4$, we obtain that
\begin{equation*}
  E(u)=E(u)-\frac{1}{4}P(u)=\frac{a}{4}\|\nabla u\|^{2}_{2}+(\frac{\beta_{p}}{4}-\frac{1}{p})\|u\|^{p}_{p}+\frac{1}{12}\|u\|^{6}_{6}\geq \frac{a\delta}{4}>0.
\end{equation*}
Thus, $\gamma_{c}>0$.
\end{proof}

\section{Proof of Theorem \ref{main reslut}}
Choosing a $PS$ sequence $\{v_{n}\}$ as in Proposition \ref{PPS} and applying Lagrange multipliers rule to $(3)$ of Proposition \ref{PPS},  there exists a sequence $\{\lambda_{n}\}\subset\mathbb{R}$ such that
\begin{equation}\label{Lagrange}
E'(v_{n})-\lambda_{n}\Psi'(v_{n})\rightarrow 0\quad {\rm in}\quad H^{-1}.
\end{equation}
\textbf{Step 1:} We claim that $v_{n}$ is bounded in $H$ and up to a subsequence, $v_{n}\rightharpoonup v$ in $H$. Since $P(v_{n})\rightarrow 0$, we have
\begin{equation*}
  E(v_{n})+o_{n}(1)=E(v_{n})-\frac{1}{4}P(v_{n})=\frac{a}{4}\|\nabla v_{n}\|^{2}_{2}+(\frac{\beta_{p}}{4}-\frac{1}{p})\|v_{n}\|^{p}_{p}+\frac{1}{12}\|v_{n}\|^{6}_{6}.
\end{equation*}
Then, using the fact $p\beta_{p}\geq 4$, we deduce that
\begin{equation*}
 \frac{a}{4}\|\nabla v_{n}\|^{2}_{2}\leq\gamma_{c}+o_{n}(1).
\end{equation*}
Thus, $v_{n}$ is bounded in $H$. Since the embedding $H_{r}^{1}(\mathbb{R}^{3})\hookrightarrow L^{q}(\mathbb{R}^{3})$ is compact for $q\in(2,6)$, we deduce that there exists $v\in H_{r}^{1}(\mathbb{R}^{3})$ such that, up to a subsequence, $v_{n}\rightharpoonup v$ in $H$.

\noindent\textbf{Step 2:} We will prove that, up to a subsequence, ${\lambda_{n}}\rightarrow\lambda<0.$

Since  $v_{n}$ is bounded in $H$, by (\ref{Lagrange}), we know that
\begin{equation*}
  E'(v_{n})v_{n}-\lambda_{n}\Psi'(v_{n})v_{n}=o_{n}(1).
\end{equation*}
Therefore,
\begin{equation}\label{lamda}
  \lambda_{n}\|v_{n}\|^{2}_{2}=a\|\nabla v_{n}\|^{2}_{2}+b\|\nabla v_{n}\|^{4}_{2}-\|v_{n}\|^{p}_{p}-\|v_{n}\|^{6}_{6}.
\end{equation}
Using the fact that $\|v_{n}\|^{2}_{2}=c^{2}$ and $ \{v_{n}\}$ is bounded in $H$, we deduce that $\{\lambda_{n}\}$ is bounded.  Hence, up to a subsequence, ${\lambda_{n}}\rightarrow\lambda\in\mathbb{R}.$  Putting $P(v_{n})\rightarrow 0$ into (\ref{lamda}), we obtain that
\begin{align}\label{L-2L-p}
  \lambda c^{2}&=\lim_{n\rightarrow \infty}a\|\nabla v_{n}\|^{2}_{2}+b\|\nabla v_{n}\|^{4}_{2}-\|v_{n}\|^{p}_{p}-\|v_{n}\|^{6}_{6}\\
  &=(\beta_{p}-1)\|v\|^{p}_{p}\leq0.
\end{align}
Hence, $\lambda=0$ if and only if $v\equiv0$. Therefore, we only need to prove that  $v\neq0$.
We assume by contradiction that $v\equiv0$.  Up to a subsequence, let $\|\nabla v_{n}\|^{2}_{2}\rightarrow l\in\mathbb{R}.$ Since $P(v_{n})\rightarrow 0$ and $v_{n}\rightarrow 0$ in  $L^{p}(\mathbb{R}^{3}),$  we have
\begin{equation*}
 \|v_{n}\|^{6}_{6}\rightarrow al+bl^{2}.
\end{equation*}
By (\ref{S}), one has
\begin{equation*}
  (al+bl^{2})^{\frac{1}{3}}\leq \frac{l}{S}.
\end{equation*}
So, we have
\begin{equation*}
l=0 \ \ {\rm or}\ \ l\geq \frac{bS^{3}}{2}+\sqrt{\frac{b^{2}S^{6}}{4}+{aS^{3}}}.
\end{equation*}
If $l> 0$, then we have
\begin{align*}
  m(c)+o_{n}(1)&=  E(v_{n})=E(v_{n})-\frac{1}{6}P(v_{n})+o_{n}(1)\\
   &= \frac{a}{3}\|\nabla v_{n}\|^{2}_{2}+\frac{b}{12}\|\nabla v_{n}\|^{4}_{2}+o_{n}(1)\\
  & = \frac{a}{3}l+\frac{b}{12}l^{2}+o_{n}(1).
\end{align*}
Hence $m(c)=\frac{a}{3}l+\frac{b}{12}l^{2}$. On the other hand, by $l\geq \frac{bS^{3}}{2}+\sqrt{\frac{b^{2}S^{6}}{4}+{aS^{3}}},$ we infer that
\begin{align*}
  m(c)&= \frac{a}{3}l+\frac{b}{12}l\\
   &\geq \frac{a}{3}\left(\frac{bS^{3}}{2}+\sqrt{\frac{b^{2}S^{6}}{4}+{aS^{3}}}\right)
   +\frac{b}{12}\left(\frac{bS^{3}}{2}+\sqrt{\frac{b^{2}S^{6}}{4}+{aS^{3}}}\right)^{2}\\
   &=\frac{abS^{3}}{4}+\frac{b^{3}S^{6}}{24}+\frac{(4aS+b^{2}S^{4})^{\frac{3}{2}}}{24}=\gamma^{*}_{c}.
\end{align*}
By Lemma \ref{m-p-l-upbdd}, this contradicts to $m(c)=\gamma_{c}<\gamma^{*}_{c}$. If $l=0,$ then we have
\begin{equation*}
\|\nabla v_{n}\|^{2}_{2}\rightarrow 0, \quad \|v_{n}\|^{6}_{6}\rightarrow 0,
\end{equation*}
 implying that  $E(v_{n})\rightarrow 0$, which is a contradiction to $\gamma_{c}>0$.

\noindent\textbf{Step 3:} Since $\lambda <0$, we can define an equivalent norm of $H$
\begin{equation*}
  \|u\|^{2}=a\int_{\mathbb{R}^{3}}|\nabla u(x)|^{2}dx-\lambda\int_{\mathbb{R}^{3}}|u(x)|^{2}dx.
\end{equation*}
Up to a subsequence, let $\lim_{n\rightarrow\infty}\|\nabla v_{n}\|_{2}^{2}=A^{2}>0$. Then $v$ satisfies
\begin{equation}\label{weak-solu}
 a\int_{\mathbb{R}^{3}}\nabla v\nabla\phi+bA^{2}\int_{\mathbb{R}^{3}}\nabla v\nabla\phi-\lambda\int_{\mathbb{R}^{3}}v\phi-\int_{\mathbb{R}^{3}}|v|^{p-2}v\phi
 -\int_{\mathbb{R}^{3}}v^{5}\phi=0
\end{equation}
for any $\phi\in H^{1}_{0}(\mathbb{R}^{3})$. Let
\begin{equation*}
  J_{\lambda}(u)= \frac{a}{2}\int_{\mathbb{R}^{3}}|\nabla u|^{2}-\frac{\lambda}{2}\int_{\mathbb{R}^{3}}| u|^{2}+\frac{bA^{2}}{2}\int_{\mathbb{R}^{3}}|\nabla u|^{2}-\frac{1}{p}\int_{\mathbb{R}^{3}}|u|^{p}-\frac{1}{6}\int_{\mathbb{R}^{3}}|u|^{6}.
\end{equation*}
Then we have $ J'_{\lambda}(v)=0$ and $\{v_{n}\}$ is a $PS$ sequence of $J_{\lambda}(u)$. The Pohozaev identity associated with (\ref{weak-solu}) is
\begin{equation*}
  P_{\lambda}(u)=\frac{a}{2}\int_{\mathbb{R}^{3}}|\nabla u|^{2}-\frac{3\lambda}{2}\int_{\mathbb{R}^{3}}| u|^{2}+\frac{bA^{2}}{2}\int_{\mathbb{R}^{3}}|\nabla u|^{2}-\frac{3}{p}\int_{\mathbb{R}^{3}}|u|^{p}-\frac{1}{2}\int_{\mathbb{R}^{3}}|u|^{6}.
\end{equation*}
Hence, we have
\begin{align}\label{J-re}
 J_{\lambda}(v)&=J_{\lambda}(v)-\frac{1}{3}P_{\lambda}(v)=\frac{a+bA^{2}}{3}\int_{\mathbb{R}^{3}}|\nabla v|^{2}\nonumber\\
 &=\frac{a+bA^{2}}{4}\int_{\mathbb{R}^{3}}|\nabla v|^{2}+\frac{1}{4}(J_{\lambda}(v)-P_{\lambda}(v)).
\end{align}
Let $w_{n}=v_{n}-v$. Then we have $w_{n}\rightharpoonup  0$ in $H$ and $w_{n}\rightharpoonup  0$ in $L^{q}(\mathbb{R}^{3})$ for $2<q<6$. Using Brezis-Lieb lemma,
\begin{equation}\label{B-L-1}
  \begin{cases}
\|w_{n}\|^{2}=\|v_{n}\|^{2}-\|v\|^{2}+o_{n}(1);\\
\|w_{n}\|_{6}^{6} =\|v_{n}\|_{6}^{6}-\|v\|_{6}^{6}+o_{n}(1);   \\
A^{2}=\|\nabla v_{n}\|_{2}^{2}+o_{n}(1)=\|\nabla w_{n}\|_{2}^{2}+\|\nabla v\|_{2}^{2}+o_{n}(1).
\end{cases}
\end{equation}
Since $\{v_{n}\}$ is a $PS$ sequence of $J_{\lambda}$, it follows (\ref{B-L-1}) that
\begin{align*}
  o_{n}(1)&=\langle J'_{\lambda}(v_{n}),v_{n}\rangle\\
  &=\langle J'_{\lambda}( v),v\rangle+(a+bA^{2})\int_{\mathbb{R}^{3}}|\nabla w_{n}|^{2}-\lambda\int_{\mathbb{R}^{3}}|w_{n}|^{2}-\int_{\mathbb{R}^{3}} |w_{n}|^{6}+ o_{n}(1)\\
  &=\|w_{n}\|^{2}+b\|\nabla w_{n}\|_{2}^{4}+b\int_{\mathbb{R}^{3}}|\nabla w_{n}|^{2}\int_{\mathbb{R}^{3}}|\nabla v|^{2}-\|w_{n}\|_{6}^{6}+ o_{n}(1),
\end{align*}
which yields
\begin{equation}\label{B-L-2}
 \|w_{n}\|^{2}+b\|\nabla w_{n}\|_{2}^{4}+b\int_{\mathbb{R}^{3}}|\nabla w_{n}|^{2}\int_{\mathbb{R}^{3}}|\nabla v|^{2}-\|w_{n}\|_{6}^{6}=o_{n}(1).
\end{equation}
Up to a subsequence, we assume that
\begin{equation*}
  \|w_{n}\|^{2}\rightarrow l_{1}\geq 0,\quad b\|\nabla w_{n}\|_{2}^{4}+b\int_{\mathbb{R}^{3}}|\nabla w_{n}|^{2}\int_{\mathbb{R}^{3}} |\nabla v|^{2}\rightarrow l_{2}\geq 0 ,\quad|w_{n}\|_{6}^{6}\rightarrow l_{3}\geq 0.
\end{equation*}
Then we have $l_{1}+l_{2}=l_{3}$. If $l_{1}=0$, we obtain $v_{n}\rightarrow v $ in $H$,  which completes the proof.
If $l_{1}>0$, by Sobolev inequality \ref{S}, we have
\begin{equation*}
 l_{1}\geq aS(l_{1}+l_{2})^{\frac{1}{3}}\quad {\rm  and} \quad  l_{2}\geq bS^{2}(l_{1}+l_{2})^{\frac{2}{3}}.
\end{equation*}
By Lemma \ref{system}, we have
\begin{equation}\label{l-esti}
 l_{1}\geq \frac{abS^{3}+a\sqrt{b^{2}S^{6}+4aS^{3}}}{2}\quad {\rm  and} \quad l_{2}\geq abS^{3}+\frac{b^{3}S^{6}+b^{2}S^{3}\sqrt{b^{2}S^{6}+4aS^{3}}}{2}.
\end{equation}
From (\ref{B-L-1}) and (\ref{B-L-2}), we deduce that
\begin{align}\label{B-L-3}
 J_{\lambda}(v_{n})&=J_{\lambda}(v)+\frac{(a+bA^{2})}{2}\int_{\mathbb{R}^{3}}|\nabla w_{n}|^{2}-\frac{\lambda}{2}\int_{\mathbb{R}^{3}}|w_{n}|^{2}-\frac{1}{6}\int_{\mathbb{R}^{3}} |w_{n}|^{6}+ o_{n}(1)\nonumber\\
 &=J_{\lambda}(v)+\frac{1}{2}\|w_{n}\|^{2}+\frac{b}{4}\left(\|\nabla w_{n}\|_{2}^{4}+\int_{\mathbb{R}^{3}}|\nabla w_{n}|^{2}\int_{\mathbb{R}^{3}}|\nabla v|^{2}\right)\nonumber\\
 &-\|w_{n}\|_{6}^{6}+ \frac{bA^{2}}{4}\int_{\mathbb{R}^{3}}|\nabla w_{n}|^{2}+o_{n}(1)\nonumber\\
 &=J_{\lambda}(v)+\frac{1}{3}\|w_{n}\|^{2}+\frac{b}{12}\left(\|\nabla w_{n}\|_{2}^{4}+\int_{\mathbb{R}^{3}}|\nabla w_{n}|^{2}\int_{\mathbb{R}^{3}}|\nabla v|^{2}\right)\nonumber\\
 &+ \frac{bA^{2}}{4}\int_{\mathbb{R}^{3}}|\nabla w_{n}|^{2}+o_{n}(1).
\end{align}
On the the other hand, since  $E(v_{n})\rightarrow \gamma_{c}$, we have
\begin{equation}\label{J-level}
 J_{\lambda}(v_{n})= \gamma_{c}-\frac{\lambda}{2}\|v_{n}\|_{2}^{2}+\frac{b}{4}A^{2}+o_{n}(1).
\end{equation}
In view of (\ref{J-re}), (\ref{l-esti}), (\ref{B-L-3})  and (\ref{J-level}), we infer that
\begin{align*}
   \gamma_{c}-\frac{\lambda}{2}\|v_{n}\|_{2}^{2}+\frac{b}{4}A^{2}&=\frac{1}{3}\|w_{n}\|^{2}+\frac{b}{12}\left(\|\nabla w_{n}\|_{2}^{4}+\int_{\mathbb{R}^{3}}|\nabla w_{n}|^{2}\int_{\mathbb{R}^{3}}|\nabla v|^{2}\right)\\
 &+J_{\lambda}(v)+ \frac{bA^{2}}{4}\int_{\mathbb{R}^{3}}|\nabla w_{n}|^{2}+o_{n}(1) \\
  &\geq\frac{abS^{3}}{4}+\frac{b^{3}S^{6}}{24}+\frac{(4as+b^{2}S^{4})^{\frac{3}{2}}}{24}+J_{\lambda}(v)+ \frac{bA^{2}}{4}\int_{\mathbb{R}^{3}}|\nabla w_{n}|^{2}+o_{n}(1) \\
  &=\gamma^{*}_{c}+\frac{a+bA^{2}}{4}\int_{\mathbb{R}^{3}}|\nabla v|^{2}+\frac{bA^{2}}{4}\int_{\mathbb{R}^{3}}|\nabla w_{n}|^{2}\\
  &+\frac{\lambda}{4}\|v\|_{2}^{2}+\frac{1}{2p}\|v\|_{p}^{p}+\frac{1}{12}\|v\|_{6}^{6}+o_{n}(1) \\
  &\geq \gamma^{*}_{c}+\frac{b}{4}A^{2}+\frac{\lambda}{4}\|v\|_{2}^{2}+\frac{1}{2p}\|v\|_{p}^{p}+o_{n}(1),
\end{align*}
which implies that
\begin{align*}
  \gamma_{c}&\geq\gamma^{*}_{c}+\frac{\lambda}{4}\|v\|_{2}^{2}
  +\frac{1}{2p}\|v\|_{p}^{p}+\frac{\lambda}{2}\|v_{n}\|_{2}^{2}+o_{n}(1)\\
  &\geq\gamma^{*}_{c}+\frac{3\lambda}{4}\|v_{n}\|_{2}^{2}
  +\frac{1}{2p}\|v\|_{p}^{p}+o_{n}(1)\\
  &=\gamma^{*}_{c}+\frac{3\lambda}{4}c^{2}
  +\frac{1}{2p}\|v\|_{p}^{p}\\
  &=\gamma^{*}_{c}+\left(\frac{3}{4}(\beta_{p}-1)+\frac{1}{2p}\right)\|v\|_{p}^{p}\geq\gamma^{*}_{c},
\end{align*}
where the last inequality is  from (\ref{L-2L-p}) and $p\beta_{p}\geq 4$. This is a contradiction to $\gamma_{c}<\gamma^{*}_{c}$. So we obtain $v_{n}\rightarrow v $ in $H$. And $E(v)=\inf_{u\in V(c)} E(u)$ follows from lemma \ref{gammac-mc}. \qed


\end{document}